\documentclass{elsarticle}
\usepackage[framemethod=default]{mdframed}
\usepackage[english]{babel}
\usepackage[misc]{ifsym}
\usepackage{geometry}
\usepackage{listings} 
\usepackage{algorithmic,algorithm}
\usepackage{multirow}
\usepackage{enumerate}
\usepackage{multirow, array}
\usepackage{datetime}
\usepackage{undertilde}
\usepackage[justification=centering]{caption}

\usepackage{booktabs}
\usepackage{tabularx}

\usepackage{amscd}
\usepackage{graphicx}
\usepackage{lscape}
\usepackage{indentfirst}
\usepackage{latexsym}
\usepackage{amsmath, amsfonts, amssymb,mathrsfs}
\usepackage{verbatim}
\usepackage{babel,blindtext}
\usepackage[numbers]{natbib}
\usepackage{amsthm}

\usepackage[colorlinks,linkcolor=red,anchorcolor=blue,citecolor=magenta]{hyper ref}

\newtheorem{lemma}[equation]{Lemma}
\newtheorem{theorem}[equation]{Theorem}

\newtheorem{remark}[equation]{Remark}

\newcommand{\mcal}{\mathcal}

\numberwithin{equation}{section} \numberwithin{table}{section}
\numberwithin{figure}{section}
\numberwithin{algorithm}{section}
\everymath{\displaystyle}

\begin{document}
\begin{frontmatter}
\title{A two-level stochastic collocation method for semilinear elliptic equations with random coefficients}

\author[SJU]{Luoping Chen}
\ead{clpchenluoping@163.com}
\address[SJU]{School of Mathematics, Southwest Jiaotong University, Chengdu 611756, China}

\author[PNNL]{Bin Zheng\corref{cor1}}
\ead{bin.zheng@pnnl.gov}
\address[PNNL]{Advanced Computing, Mathematics \& Data Division, Pacific Northwest National Laboratory, Richland, WA 99352, USA}

\author[Purdue]{Guang Lin}
\ead{guanglin@purdue.edu}
\address[Purdue]{Department of Mathematics, School of Mechanical Engineering, Purdue University, West Lafayette, IN 47907, USA}

\author[WSU]{Nikolaos Voulgarakis}
\ead{nvoul@tricity.wsu.edu}
\address[WSU]{Department of Mathematics, Washington State University, Tri-Cities, Richland, WA, 99354, USA}

\cortext[cor1]{Corresponding author}

\begin{abstract}
In this work, we propose a novel two-level discretization for solving semilinear elliptic equations with random coefficients. Motivated by the two-grid method for deterministic partial differential equations (PDEs) introduced by Xu \cite{xu1994novel}, our two-level stochastic collocation method utilizes a two-grid finite element discretization in the physical space and a two-level collocation method in the random domain. In particular, we solve semilinear equations on a coarse mesh $\mathcal{T}_H$ with a low level stochastic collocation (corresponding to the polynomial space $\mathcal{P}_{\boldsymbol{P}}$) and solve linearized equations on a fine mesh $\mathcal{T}_h$ using high level stochastic collocation (corresponding to the polynomial space $\mathcal{P}_{\boldsymbol{p}}$). We prove that the approximated solution obtained from this method achieves the same order of accuracy as that from solving the original semilinear problem directly by stochastic collocation method with $\mathcal{T}_h$ and $\mathcal{P}_{\boldsymbol{p}}$. The two-level method is computationally more efficient than the standard stochastic collocation method for solving nonlinear problems with random coefficients. Numerical experiments are provided to verify the theoretical results.
\end{abstract}


\end{frontmatter}

\section{Introduction}
Stochastic partial differential equations (SPDEs), especially nonlinear SPDEs, provide mathematical models for the quantification of uncertainties in many complex physical and engineering applications. Some examples include the propagation of uncertainties associated with input parameters (such as the coefficients, forcing terms, boundary conditions, geometry of the domain etc.) to certain output quantities of interests, see e.g., flow in heterogeneous porous media \cite{ma2011stochastic}, thermo-fluid processes \cite{knio2001stochastic, le2002stochastic}, flow-structure interactions \cite{xiu2002stochastic}. More applications of the nonlinear SPDEs in physics and mechanics can be found in \cite{adomian1989nonlinear, bellomo1987nonlinear}. 

Numerical methods dealing with SPDEs can be roughly categorized as either intrusive or non-intrusive types. The stochastic Galerkin (SG) method based on the polynomial chaos expansion \cite{xiu2002wiener, xiu2002stochastic, xiu2003modeling, babuska2004galerkin, najm2009uncertainty, matthies2005galerkin, zhen2014} is considered as an intrusive method since it results in coupled systems which cannot be solved directly by the corresponding deterministic solvers. The SG method applies Galerkin projection to discretize the stochastic space and uses standard finite element discretization in the physical space. It is often advantageous over non-intrusive approaches in terms of the computational efficiency when efficient solvers are available. It provides an exponentially convergent approximation when the solution of stochastic problem is smooth with respect to the random variables. However, it is reported that for some nonlinear problems, the stochastic Galerkin method may not be as efficient as non-intrusive stochastic methods \cite{yaonumerical}. The Monte Carlo (MC) method \cite{shinozuka1972monte, caflisch1998monte, matthies2005galerkin, yaonumerical} is the most widely used non-intrusive method based on the sampling techniques. MC method is attractive because its convergence is independent of the stochastic dimension. On the other hand, its rate of convergence is rather slow, proportional to $1/\sqrt{N}$ with $N$ being the number of samples. Another non-intrusive type method, stochastic collocation (SC) method \cite{xiu2005high, hosder2006non, nobile2008sparse, ma2009adaptive, babuvska2010stochastic, yaonumerical}, has recently gained popularity. The stochastic collocation method shares the same exponential convergence property as the stochastic Galerkin method. Moreover, stochastic collocation method only requires solving the deterministic problem on a set of collocation points, hence existing efficient and robust solvers for these problems are applicable. Both the stochastic Galerkin and stochastic collocation method suffer from the curse of dimensionality. To address the problem of the curse of dimensionality, methods based on multiscale finite element method in physical space and sparse grid collocation method in stochastic space have been proposed \cite{Aarnes:2008aa, Ma:2011aa}. In \cite{Zhang:2015aa} a multiscale data-driven stochastic method is proposed to reduce both the stochastic and the physical dimensions of the solution.

The purpose of this study is to improve the efficiency of the stochastic collocation method for solving semilinear SPDEs. Our motivation comes from the two-grid finite element discretization proposed by Xu \cite{xu1994novel, xu1996two} for the nonsymmetric, indefinite and nonlinear elliptic problems. The main idea of two-grid method is based on the observation that a very coarse grid space is sufficient for some nonsymmetric, indefinite and/or nonlinear problems that are dominated by their symmetric, positive and/or linear parts.  Later, the method has been applied to solve semilinear elliptic eigenvalue problems \cite{xu2001two, chien2006two}, nonlinear parabolic differential equations \cite{dawson1994two, dawson1998two, chen2003two, chen2007analysis}, Navier-Stokes equations \cite{layton1995two, layton1998two, ammi1994nonlinear, utnes1997two, girault2001two}, magnetohydrodynamics system \cite{layton1997two}, etc.

In order to generalize the two-grid technique for solving semilinear SPDEs, we shall utilize two meshes in the physical domain and two levels of collocation points in the random domain. Furthermore, to minimize the computational cost, we use fine mesh for spatial discretization when approximating the stochastic variables with high order polynomial space, and use coarse mesh in spatial space with low order polynomial space in stochastic space. More precisely, our method consists of two steps, i.e., we first solve nonlinear problems using the coarse mesh and low level stochastic collocation, then solve a corresponding linearized problems on the fine mesh with high level stochastic collocation. The resulting two-level discretization method is computationally more efficient. Moreover, we prove that the solution obtained from two-level approach has the same order of accuracy as that from solving nonlinear problems directly using fine mesh and high level collocation points. We verify the theoretical results by several numerical examples.

The rest of the paper is organized as follows. In Section \ref{sec:preliminary}, we introduce the model problem and some notations. The two-level method is described in detail in Section \ref{sec:algorithmandconvergence}. In Section {\ref{sec:convergence}}, we estimate the error of the approximated solution. Finally, in Section \ref{sec:numericalresult}, numerical experiments are given to verify the theoretical results.

\section{Model problem and weak formulation}
\label{sec:preliminary}
In this work, we investigate the following semilinear elliptic problem with random coefficient
\begin{equation}\label{model-problem}
\left\{
\begin{aligned}
-\nabla\cdot(a(\omega,x)\nabla u(\omega,x)) + f(\omega,x,u(\omega,x))&=0,\quad \; x\in D,\\
u(\omega,x) &= 0, \quad x\in\partial D,
\end{aligned}
\right.
\end{equation}
where $D\subset\mathbb{R}^d$ is a bounded domain, $\partial D$, the boundary of $D$, is either smooth or convex and piecewise smooth, the diffusion coefficient $a$ is a real-valued random field defined on $D$, i.e., for each $x\in D$, $a(\cdot,x):\Omega\rightarrow \mathbb{R}$ is a random variable with respect to a suitable probability space $(\Omega,\mathscr{F},\mathscr{P})$. Here $\Omega$ is the set of elementary events, $\mathscr{F}$ is the $\sigma$-algebra and $\mathscr{P}: \Omega\rightarrow [0,1]$ is a probability measure. We assume that $a$ is bounded and uniformly coercive, i.e., there exist $a_{min}, a_{max}\in (0,\infty)$, such that
\begin{align}\label{coercevity}
\mathscr{P}(\omega\in\Omega: a(\omega,x)\in [a_{min}, a_{max}],\forall x\in\bar{D}) = 1.
\end{align}
Here, we also assume $f(\omega,x,u(\omega,x))$ is sufficiently smooth. For brevity, we shall drop the dependence of variables $\omega, x$ in $f(\omega, x, u)$ in the following exposition. We also note that here and later in this paper the gradient operator, $\nabla$, always represents differentiation with respect to $x$ only. The model problem (\ref{model-problem}) is a prototype stationary reaction-diffusion problem that can be found in many chemical and biological applications. For example, it appears in the semi-discretization in time of the nonlinear stochastic reaction-diffusion problem modeling the conversion of starch into sugars in growing apples \cite{Rosseel:2012aa}.

We introduce some notations which will be used later. Let $(\cdot,\cdot)$ be the inner product of $\mathcal L^2(D)$. $\mathcal W^{m,q}(D)$ denotes the standard Sobolev space with norm $\|\cdot\|_{m,q}$ given by ${\textstyle{\|v\|^q_{m,q} = \sum_{|\alpha|\leq m}\|\frac{\partial^{\alpha}v}{\partial x^{\alpha}}\|_{\mathcal L^q}^q}}$, when $q=2$, we denote $\mathcal H^m(D)= \mathcal W^{m,2}(D)$. Let $\mathcal H_0^1(D)$ be the subspace of $\mathcal H^1(D)$ consisting of all the functions with vanishing trace on $\partial D$. $\|\cdot\|_m = \|\cdot\|_{m,2}$ and $\|\cdot\| = \|\cdot\|_{0,2}$. We need the following well-known Sobolev inequalities in Section {\ref{sec:convergence}} 
\begin{align}\label{sobolevinequality}
\|u\|_{0,q}\lesssim\|u\|_1\;  (d=2\ {\rm and}\ 1\leq q<\infty)\;\;{\rm and}\;\;\|u\|_{0,6}\lesssim \|u\|_1 \;(d=3),
\end{align}
where the notation ``$\lesssim$" is equivalent to ``$\leq C$" for some positive constant $C$.
We assume that problem (\ref{model-problem}) has at least one solution $u(\omega,x)\in \mathcal H_0^1(D)\cap \mathcal H^2(D)$ for each parameter $\omega\in\Omega$.

To introduce the stochastic discretization, we first approximate the input random field $a(\omega,x)$ by a truncated Karhunen-Lo\`eve (KL) expansion
\begin{align*}
a(\omega,x)\approx a_N(\omega,x) &=a_N(Y_1(\omega),Y_2(\omega),\cdots,Y_N(\omega),x)\\
&= \bar a(x)+\sum_{n=1}^N\sqrt{\lambda_n}b_n(x)Y_n(\omega),
\end{align*}
where $\bar{a}(x)$ is the mean value of $a(\omega,x)$, $(Y_{1},Y_2,\cdots,Y_N)$ are uncorrelated and identically distributed random variables with zero mean and unit variance. For simplicity, we assume that $(Y_{1},Y_2,\cdots,Y_N)$ are independent and $\{\rho_n\}_{n=1}^N$ are the probability density functions of the random variables $\{Y_n\}_{n=1}^N$. $\lambda_1\geq\lambda_2\geq\cdots\geq\lambda_i\geq\cdots\geq 0$ and $\{b_n(x)\}_{n=1}^N\subset \mathcal L^2(D)$ are the eigenvalues and eigenfunctions of the symmetric positive semidefinite Fredholm operator $C_a: \mathcal L^2(D)\rightarrow \mathcal L^2(D)$ defined  by
\begin{align*}
(C_ag)(x) = \int_{D}Cov_a(x,x')g(x')dx',
\end{align*}
with $Cov_a$ being a given continuous covariance function. The truncated KL expansion is optimal in the sense that it obtains the smallest mean square error among all approximations of $a$ in $N$ uncorrelated random variables \cite{ghanem1991stochastic}.

Let $\Gamma_n=Y_n(\Omega)$ be the image of $Y_n$, $\Gamma = \Pi_{n=1}^N\Gamma_n$. The random variables $[Y_1,Y_2,\cdots, Y_N]$ have a joint probability density function $\rho=\Pi_{n=1}^N\rho_n$. By Doob-Dynkin's Lemma {\cite{oksendal2003stochastic}}, the solution $u(\omega,x)$ can be represented by $u$($Y_1(\omega)$,$Y_2(\omega)$,$\cdots$,$Y_N(\omega)$,$x$).  Let 
$(\Gamma,\mathcal{B}^N,\rho dy)$ be a probability space with $\mathcal{B}^N$ being the $\sigma$-algebra associated with the set of outcomes $\Gamma$. The expectation of a random variable $\textstyle{\mu(y)\in (\Gamma,\mathcal{B}^N,\rho dy)}$ is ${\textstyle{E(\mu(y)) = \int_{\Gamma}\mu(y)\rho(y)dy}}$ and variance is $Var(\mu(y))$ $=$ ${\textstyle{\int_{\Gamma}\mu^2(y)\rho(y)d{y}}-}$ ${\textstyle{\left[\int_{\Omega}\mu(y)\rho(y)d{y}\right]^2}}$. Thus, after replacing the diffusion coefficient $a$ by the truncated KL expansion $a_N$, (\ref{model-problem}) can be written as the following parametrized problem with $N$-dimensional parameter
\begin{equation}\label{KLmodel-problem}
\left\{
\begin{aligned}
-\nabla\cdot(a_N(y,x)\nabla u(y,x))+f(u(y,x)) &=0,\quad x\in D,\\
u(y,x)&=0,\quad x\in\partial D.
\end{aligned}
\right.
\end{equation}
For a.e. $y\in\Gamma$ (here and in what follows, a.e. stands for `almost everywhere'), we assume that $f$ is sufficiently smooth, the problem (\ref{KLmodel-problem}) has at least one solution $u(y,\cdot)\in \mathcal H^1_0(D)\cap \mathcal H^2(D)$, and the linearized operator $L_v:=-\nabla\cdot(a_N\nabla) + f'(v)$ is nonsingular for $v\in H^1_0(D)\cap L^\infty(D)$. As a result of this assumption, for a.e. $y\in\Gamma$, $L_v: \mathcal H^2(D)\cap \mathcal H_0^1(D)\mapsto \mathcal L^2(D)$ is a bijection and satisfies
\begin{equation*}
\|w\|_2\leq C\|L_vw\|, \quad \forall w\in \mathcal H^2(D)\cap \mathcal H_0^1(D).
\end{equation*}

The use of truncated Karhunen-Lo\`eve expansion introduces a modeling error when comparing with the original problem (\ref{model-problem}). In the following, we focus on the study of the model problem (\ref{KLmodel-problem}) and neglecting this truncation error.  

We denote $\mathcal L^2_{\rho}(\Gamma)$ the Hilbert space with inner product 
$$
(f,g)_{\mathcal L^2_{\rho}} = \int_\Gamma f(y)g(y)\rho(y)dy, \quad \forall f,g\in \mathcal L^2_{\rho}(\Gamma),
$$
and introduce the following tensor product spaces
 \begin{align*}
 \mathcal L^2_{\rho}(\Gamma)\otimes V &=  \{u\,|\,u(y,\cdot)\in V,\  {\rm a.e.}\ {\rm in}\  \Gamma, \ {\rm and}\ u(\cdot,x)\in \mathcal L_{\rho}^2(\Gamma), \ {\rm a.e.}\ {\rm in}\ D\},
 \end{align*}
with $V$ a Hilbert space and the inner products defined by
\begin{align*}
(u,v)_{\mathcal L^2_{\rho}(\Gamma)\otimes V} &= E[(u, v)_V].
\end{align*}
We assume that problem (\ref{KLmodel-problem}) has at least one solution $u\in \mathcal L_{\rho}^2(\Gamma)\otimes (\mathcal H^1_0(D)\cap \mathcal H^2(D))$.

We also denote $\mathcal{L}^q_\rho(\Gamma)$ a Banach space with norm
$$
\|f\|_{\mathcal{L}^q_\rho(\Gamma)}
=\left(\int_\Gamma|f(y)|^q\rho(y)\mathrm{d}y\right)^{1/q},
$$
and define the following tensor product spaces
 \begin{align*}
 \mathcal L^q_{\rho}(\Gamma)\otimes V = \{u(y,x):\Gamma\times D\rightarrow\mathbb{R}|u(y,\cdot)\in V, \  {\rm a.e.}\ {\rm in}\  \Gamma, \ {\rm and}\ u(\cdot,x)\in \mathcal L_{\rho}^q(\Gamma), \ {\rm a.e.}\ {\rm in}\ D\},
 \end{align*}
with $V$ a Banach space and the tensor norm defined by
  \begin{align*}
\|u\|^q_{\mathcal L^q_{\rho}(\Gamma)\otimes V} = E(\|u\|^q_{V}).
 \end{align*}

For convenience, we denote $\mathcal V_{\rho} := \mathcal L^2_{\rho}(\Gamma)\otimes \mathcal H^1_0(D)$ and use notation $u(y)$ in the following whenever we want to highlight the dependence on the parameter $y$.

The weak formulation of problem (\ref{KLmodel-problem}) is to find $u\in\mathcal V_{\rho}$ such that
\begin{equation}\label{KLweakform}
\int_Da_N(y)\nabla u(y)\nabla wdx +\int_Df(u(y))wdx = 0,\ \forall w\in \mathcal H^1_0(D), \rho-{\rm a.e.\ in}\  \Gamma.
\end{equation}
and the weak formulation for the linearized problem of (\ref{KLmodel-problem}) can be represented as: for some $v\in\mathcal{L}^2_{\rho}(\Gamma)\otimes \mathcal{W}^{1,p}(D)$, find $\bar{u}\in\mathcal V_{\rho}$ such that
\begin{equation}\label{LKLweakform}
\int_Da_N\nabla \bar u\nabla wdx +\int_Df'(v)\bar uwdx = \int_D(-f(v)+f'(v)v)wdx,\ \forall w\in \mathcal H^1_0(D), \rho-{\rm a.e.\ in}\  \Gamma.
\end{equation}

Following \cite{babuvska2010stochastic}, we make an assumption that the coefficient $a_N$ and $f(u)$ admit a smooth extension on the $\rho$-zero measure sets. Then, equation (\ref{KLweakform}), (\ref{LKLweakform}) can be extended $a.e.$ in $\Gamma$ with respect to the Lebesgue measure.

\begin{remark}
Our two-level stochastic collocation method consists of solving the problem (\ref{KLweakform}) on a coarse mesh with low-level stochastic collocation, and solving the linearized problem (\ref{LKLweakform}) on a fine mesh with high-level stochastic collocation.
\end{remark}

\section{Two-level discretization for semilinear SPDEs}
\label{sec:algorithmandconvergence}
In this section, we first describe the stochastic collocation method following \cite{babuvska2010stochastic}. Then, we present the two-level stochastic collocation method for solving semilinear PDEs with random coefficients.

\subsection{Stochastic collocation method}
We first introduce the finite dimensional subspace $\mathcal V_{\boldsymbol{p},h}\subset \mathcal V_{\rho}$ given by $\mathcal{P}_{\boldsymbol{p}}(\Gamma)\otimes \mathcal X_h(D)$, where

\begin{itemize}
\item $\mcal{P}_{\boldsymbol{p}}(\Gamma) = \otimes_{n=1}^N\mcal{P}_{p_n}(\Gamma_n)$ is the span of the tensor product polynomials with degree at most $\boldsymbol{p} = (p_1,p_2,\cdots,p_N)$, and
$$
\mcal{P}_{p_n}(\Gamma_n) = {\rm{span}}\{y_n^m, m = 0,1,\cdots,p_n\},\; n = 1,2,\cdots, N.
$$

Therefore, the dimension of $\mcal{P}_{\boldsymbol{p}}$ is $N_{\boldsymbol{p}} = \Pi_{n=1}^{N}(p_n+1)$.
\item $\mathcal X_h(D) = \ {\rm span}\ \{\phi_1,\phi_2,\cdots, \phi_{N_{h}}\}$ is a finite element space of dimension $N_{h}$, where $\phi_1,\phi_2,\cdots, \phi_{N_h}$ are piecewise polynomials defined on a quasi-uniform triangulation $\mcal{T}_h$ with mesh size $h$.
\end{itemize}

We first introduce a semi-discrete approximation $u_{h}:\;\Gamma\rightarrow \mathcal X_h(D)$, i.e., for a.e. $y\in\Gamma$, $\forall w\in \mathcal X_h(D)$,
\begin{equation}\label{weakformula}
\int_Da_N(y)\nabla u_h(y)\cdot\nabla w dx + \int_Df(u_h(y))w dx = 0.
\end{equation}
Similarly, for a.e. $y\in\Gamma$, we also introduce the semi-discrete approximate $u^h$ of the linearized equation (\ref{LKLweakform}) satisfying
\begin{equation}\label{Lweakform}
\int_Da_N\nabla u^h\nabla wdx +\int_Df'(v)u^hwdx = \int_D(-f(v)+f'(v)v)wdx, \quad\forall w\in \mathcal{X}_h(D).
\end{equation}

Next, we collocate equation (\ref{weakformula}) on the roots of orthogonal polynomials with respect to the weight $\rho$ and build the fully discrete solution $u_{h,\boldsymbol{p}}\in\mathcal{P}_{\boldsymbol{p}}(\Gamma)\otimes \mathcal X_h(D)$ by interpolating in $y$ with the collocated solutions, i.e.
$$
u_{h,\boldsymbol{p}}(y,x) = \sum_{k=1}^{N_{\boldsymbol{p}}}u_h(\hat y_k,x)\psi_k(y),
$$
where $u_h(\hat y_k,\cdot)$ is the solution of (\ref{weakformula}) at the collocation point $\hat y_k = (y_{1,k_1},y_{2,k_2},\cdots, y_{N,k_N})$ and $\{\psi_k(y)\}_{k=1}^{N_{\boldsymbol{p}}}$ are the Lagrange basis with respect to the collocation points $\{\hat y_{k}\}_{k=1}^{N_{\boldsymbol{p}}}$. Using the Lagrange interpolation operator $\mathcal{I}_{\boldsymbol{p}}:$ $\mathcal C^{0}(\Gamma; \mathcal H_0^1(D))$$\rightarrow$$\mcal{P}_{\boldsymbol{p}}(\Gamma)\otimes \mathcal H_0^1(D)$, defined by 
$$
(\mcal{I}_{\boldsymbol{p}} v)(y) = \sum_{k=1}^{N_{\boldsymbol{p}}}v(\hat y_k)\psi_k(y),\quad \forall \ v\in \mathcal C^0(\Gamma;\mathcal H_0^1(D)),
$$
we have $u_{h,\boldsymbol{p}}= \mcal{I}_{\boldsymbol{p}}u_h$.

\subsection{Two-level stochastic collocation method}
In this subsection, we shall present a two-level discretization scheme for semilinear elliptic equations with random coefficients based on two tensor product spaces $\mcal{P}_{\boldsymbol{P}}(\Gamma)\otimes \mathcal X_H(D)$ and $\mcal{P}_{\boldsymbol{p}}(\Gamma)\otimes \mathcal X_h(D)$. The idea of the two-level method is to reduce a nonlinear SPDE problem into a linear SPDE problem by solving a nonlinear SPDE problem on a much smaller space. The method is described in detail as follows. \\

\begin{mdframed}[frametitle={Two-level discretization}]
\begin{itemize}
\renewcommand{\labelitemi}{$\bullet$}
\item \textbf{Step 1:} on the coarse mesh $\mathcal{T}_H$, we solve the semilinear equation on a small number of collocation points. More precisely, for $k=1, 2, \cdots, N_{\boldsymbol{P}}$, find $u_H(\hat y^c_k,\cdot)$ on the coarse mesh such that

\begin{align}\label{collocation-Galerkinweakform}
(a_N(\hat y^c_k,\cdot)\nabla u_H(\hat y^c_k,\cdot),\nabla w) + (f(u_H(\hat y^c_k,\cdot)),w) = 0,\quad \forall w\in \mathcal X_H(D),
\end{align}
where $\{\hat y_k^c\}_{k=1}^{N_{\boldsymbol{P}}}$ is the set of collocation points corresponding to polynomial space $\mathcal{P}_{\boldsymbol{P}}(\Gamma)$.

The approximated solution of (\ref{KLweakform}) in $\mathcal{P}_{\boldsymbol{P}}(\Gamma)\otimes\mathcal{X}_H(D)$ is given by
\begin{equation*}
\hspace{1.5cm}u_{H,\boldsymbol{P}}(y,x) = (\mcal{I}_{\boldsymbol{P}} u_H)(y) = \sum_{k=1}^{N_{\boldsymbol{P}}}u_H(\hat y^c_k,x)\psi^{\boldsymbol{P}}_k(y),
\end{equation*}
where $\{\psi^{\boldsymbol{P}}_k\}_{k=1}^{N_{\boldsymbol{P}}}$ are Lagrange basis functions of $\mathcal{P}_{\boldsymbol{P}}(\Gamma)$.

\vspace{6pt}
\item \textbf{Step 2:} We solve the following linearized problem on a larger set of collocation points. Namely, find $u^{h}(\hat y_k,\cdot)$ on the fine mesh $\mathcal{T}_h$ such that
\begin{equation}\label{collocation-Galerkin-finemesh}
\begin{aligned}
&(a_N(\hat y_k,\cdot)\nabla u^{h}(\hat y_k,\cdot),\nabla w)  + (f'(u_{H,\boldsymbol{P}}(\hat y_k,\cdot))u^{h}(\hat y_k,\cdot),w) \\
 &= (-f(u_{H,\boldsymbol{P}}(\hat y_k,\cdot))+f'(u_{H,\boldsymbol{P}}(\hat y_k,\cdot))u_{H,\boldsymbol{P}}(\hat y_k,\cdot),w), \ \forall w\in \mathcal X_h(D),
\end{aligned}
\end{equation}
where  $\{\hat y_k\}_{k=1}^{N_{\boldsymbol{p}}}$ is the set of collocation points corresponding to $\mathcal{P}_{\boldsymbol{p}}(\Gamma)$.

Finally, the two-level solution $u^{h,\boldsymbol{p}}$ is given by
$$
\hspace{1.5cm}u^{h,\boldsymbol{p}}=(\mcal{I}_{\boldsymbol{p}} u^h)(y) = \sum_{k=1}^{N_{\boldsymbol{p}}}u^h(\hat y_k,x)\psi^{\boldsymbol{p}}_k(y),
$$
where $\{\psi^{\boldsymbol{p}}_k\}_{k=1}^{N_{\boldsymbol{p}}}$ are the Lagrange basis functions of $\mathcal{P}_{\boldsymbol{p}}(\Gamma)$.
\end{itemize}
\end{mdframed}
\vspace{4pt}

We use Newton's method for the semilinear system (\ref{collocation-Galerkinweakform}), i.e., for each collocation point, starting from an initial guess $u_H^0$, and for $l=0,1,\cdots$, we solve
\begin{align}\label{newtoniteration}
(a_N\nabla u_H^{{l}+1},\nabla w) + (f(u_H^{{l}})+f'(u_H^{{l}})(u_H^{{l}+1}-u_H^{l}),w) = 0,\;\forall w\in \mathcal{X}_H(D).
\end{align}
Let $\{\phi_j^H\}_{j=1}^{N_H}$ be the finite element basis functions on triangulation $\mathcal{T}_H$, $A, J_{l}$ be the matrices whose entries are given by
\begin{align*}
A_{ij} = (a_N\nabla\phi^H_j,\nabla\phi^H_i),\quad (J_{l})_{ij}  = (f'(u_H^{{l}})\phi^H_j,\phi^H_i),
\end{align*}
and ${F}_{l}$ be the right hand side vector with $({F}_{l})_i = (-f(u_H^{l})+f'(u^{l}_H)u_H^{l},\phi^H_i)$. Then, for each collocation point, Newton iteration (\ref{newtoniteration}) can be written as
\begin{align}\label{algebraicNewton}
 U_H^{{l}+1} =  U_H^{l} +(A+J_{l})^{-1}({F}_{l}-(A+J_{l}) U_H^{l}),
\end{align}
where $\textstyle{u^{{l}+1}_H = \sum_{j=1}^{N_H}( U^{{l}+1}_H)_j\phi^H_j}$. In practice, numerical quadrature with sufficient accuracy is needed to compute $(f'(u_{H,\boldsymbol{P}})u^h,w)$.      

The advantage of the two-level discretization is that we only need to solve a small number of semilinear equations on the coarse mesh in addition to solving linearized problems on the fine mesh. In fact, we solve $N_{\boldsymbol{P}} (=\Pi_{n=1}^N(P_n+1))$ seminlinear equations on the coarse mesh and $N_{\boldsymbol{p}} (=\Pi_{n=1}^N(p_n+1))$ linerized equations on the fine mesh. From the analysis given in Section \ref{sec:convergence}, when choosing $P_n=p_n/2 \ (n=1,2,\cdots, N)$ and $H=h^{1/4}$ in the two-level discretization, the resulting approximated solution has the same order of accuracy as that obtained from the standard stochastic collocation method on mesh $\mathcal{T}_h$ and tensor-product polynomial space $\mathcal{P}_{\boldsymbol{p}}(\Gamma)$. For the standard stochastic collocation method, we need to solve $N_{\boldsymbol{p}}$ semilinear equations. Roughly speaking, this corresponds to solving $kN_{\boldsymbol{p}}$ linear equations on the fine mesh if we assume the number of Newton iteration is $k$ for each collocation point. Hence, the two-level method saves $(k-1)N_{\boldsymbol{p}}$ linear solves on the fine mesh at the expense of $kN_{\boldsymbol{P}}$ linear solvers on the coarse mesh. When $h>>H$ and $N$ is big, the computational savings is enormous. For example, if the fine mesh size $h=2^{-12}$ which gives $\text{dim}\mathcal{X}_h\approx 1.7\times 10^7$, choosing $H=h^{1/4}$ gives $\text{dim}\mathcal{X}_H\approx 49$; if the number of random variables $N=8$, the polynomial degree of each random dimension in $\mathcal{P}_{\boldsymbol{p}}$ is $p_n=4$ (for $n=1,2,\dots, N$) which gives $N_{\boldsymbol{p}}\approx 3.9\times 10^5$, choosing $P_n=p_n/2$ gives $N_{\boldsymbol{P}}\approx 6.5\times 10^3$, so the two-level method saves the solving of approximately $10^5$ linear system of equations with dimension $10^7$.  

\begin{remark}
The two-level stochastic collocation method can be parallelized by solving each realization of the coarse-grid problem independently, followed by a barrier due to the interpolation of the coarse-level solution, and then solving each realization of the fine-grid problem independently. Because the overall computational cost is dominated by solving fine-grid problems, this barrier does not have significant impact on the parallelizability of this method.
\end{remark}

\begin{remark}
In this work, we assume that the diffusion coefficient $a$ in the model problem (\ref{model-problem}) is a random field. In general, the two-level stochastic collocation method described above is applicable to model problems with uncertain boundary conditions or uncertain source terms. For such cases, the stochastic solutions can also be described by tensor product polynomials when using stochastic collocation method to propagate uncertainty from boundary conditions or source terms to solutions \cite{Le-Maitre2010aa}. The tensor product structure of the solution space allows us to apply the two-level stochastic collocation method to model problems with other sources of uncertainty.
\end{remark}

\begin{remark}
It is known that the stochastic collocation method using full-tensor product polynomials suffers from the curse-of-dimensionality. For problems with high stochastic dimensions, we may consider the Smolyak sparse grid collocation method which has the same asymptotic accuracy as full-tensor product collocation method \cite{Nobile:2008aa}. We expect that the two-level method using Smolyak sparse grid will also be efficient for nonlinear problems with high random dimensions. Future research is needed along this direction.
\end{remark}

\section{Convergence analysis}
\label{sec:convergence}
\label{subsec:converge}
We shall now derive some error estimates for the two-level discretization introduced in Section {\ref{sec:algorithmandconvergence}}. To simplify the analysis, we do not include the effect of numerical quadrature. We refer to \cite{Feistauer:1987aa,Abdulle:2012aa} for the a priori error estimates of finite element methods with numerical quadrature for nonlinear elliptic problems. 

We first give the following error estimate for the semi-discrete solution by finite element methods. 
\begin{lemma}\label{error4fem}
Let $u_h: \Gamma\rightarrow \mathcal X_h(D)$ be the semi-discrete finite element solution satisfying (\ref{weakformula}). 

Then, for $2\leq p<\infty,\  2\leq q<\infty,$
\begin{align*}
\|u-u_h\|_{\mathcal L^q_{\rho}(\Gamma)\otimes \mathcal L^p(D)}+h\|u-u_h\|_{\mathcal L^q_{\rho}(\Gamma)\otimes \mathcal W^{1,p}(D)}&\lesssim h^2\|u\|_{\mathcal L^q_{\rho}(\Gamma)\otimes \mathcal W^{2,p}(D)}
\end{align*}
and
\begin{align*}
\|u-u_h\|_{\mathcal L_{\rho}^2(\Gamma)\otimes \mathcal L^{\infty}(D)}&\lesssim h^2|\log h|(\|u\|_{\mathcal L^2_{\rho}(\Gamma)\otimes \mathcal W^{2,\infty}(D)}),\\
 \|u-u_h\|_{\mathcal L^2_{\rho}(\Gamma)\otimes \mathcal W^{1,\infty}(D)}&\lesssim h(\|u\|_{\mathcal L^2_{\rho}(\Gamma)\otimes \mathcal W^{2,\infty}(D)}).
\end{align*}
\end{lemma}
\begin{proof}
It follows directly from the result of the corresponding deterministic problem \cite{xu1994novel}.
\end{proof}

For the linearized operator $L_{v(y)} = -\nabla\cdot(a_N(y)\nabla) +f'(v(y))$,  we have the following property.

\begin{lemma}\label{quasi-linear-operator} There exists a constant $\delta >0$ such that for any given $v\in \mathcal L^2_{\rho}(\Gamma)\otimes(\mathcal H^1_0(D)\cap \mathcal L^{\infty}(D))$ with $\|u-v\|_{\mathcal L^2_{\rho}(\Gamma)\otimes \mathcal L^{\infty}(D)}\leq\delta$, and for a.e. $y\in\Gamma$ given,
\begin{itemize}
\item $L_{v(y)}:  \mathcal H^2(D)\cap \mathcal H_0^1(D)\mapsto \mathcal H^2(D)\cap \mathcal H^1_0(D)$ is bijective and there exists a constant $C = C(\delta)$, such that
$$
\|w\|_{\mathcal H^{2}(D)}\leq C(\delta)\|L_{v(y)}w\|_{\mathcal L^2(D)},\quad \forall \ w \in \mathcal H_0^1(D)\cap \mathcal H^2(D).
$$
\item If $h$ is sufficiently small, there exists a constant $c(\delta)$ such that
$$
\sup_{\chi\in \mathcal X_h(D)}\frac{A_{v(y)}(w_h,\chi)}{\|\chi\|_1}\geq c(\delta)\|w_h\|_{1},
$$
where $A_{v(y)}(w_h,\chi) = (a_N(y)\nabla w_h,\nabla\chi)+ (f'(v(y))w_h,\chi)$, and $w_h\in \mathcal X_h(D)$.
\end{itemize} 
\end{lemma}

Let $u^{h,\boldsymbol{p}}$ be the two-level solution, we have
$$
u-u^{h,\boldsymbol{p}} = (u-u_h)+(u_h-u^h) + (u^h-u^{h,\boldsymbol{p}}) = (u-u_h)+(u_h-u^h)+(u^h-\mathcal{I}_{\boldsymbol{p}}(u^h)),
$$
where $u^h$ is the semi-discrete solution satisfying the equation (\ref{Lweakform})
and $u^h-\mathcal{I}_{\boldsymbol{p}}(u^h)$ is the Lagrange interpolation error. 

Following \cite{babuvska2010stochastic}, we introduce a weighted continuous space
$$
\mathcal{C}_{\sigma}^0(\Gamma;V)\equiv\{v:\Gamma\rightarrow V, v {\rm\ is\ continuous\ in\ } y, \max_{y\in\Gamma}\|\sigma(y)v(y)\|_{V}<+\infty\},
$$
where $V$ is a Banach space with functions defined on $D$ and $\sigma(y)=\Pi_{n=1}^N\sigma_n(y_n)$ with
\begin{equation*}
\sigma_n(y_n) = 
\left\{
\begin{aligned}
&1,\qquad\qquad\qquad\qquad\qquad\ \ \ \ \ {\rm if\ \Gamma_n\ is\ bounded},\\
&e^{-\alpha_n|y_n|},\quad {\rm for\ some}\ \alpha_n>0,\ {\rm if\ \Gamma_n\ is\ unbounded},
\end{aligned}
\right.
\end{equation*}

and the following assumptions,
 \begin{itemize}
\item[(A1)]
 $f\in\mathcal{C}_{\sigma}^0(\Gamma;L^2(D))$ \\
 \item[(A2)] the joint probability density $\rho$ satisfies
 \begin{equation*}
 \rho(y)\leq C_{\rho}e^{-\Sigma_{n=1}^N(\delta_ny_n)^2},
 \quad \forall y\in\Gamma,
 \end{equation*}
 for some constant $C_{\rho}>0$ and $\delta_n$ strictly positive if $\Gamma_n$ is unbounded and zeros otherwise.
 \end{itemize}
 
The following interpolation error estimate is needed in our analysis.

 \begin{lemma}\label{interpolation-error}
 \cite{babuvska2010stochastic} There exist positive constants $r_n, n=1,2,\cdots, N$, independent of $\boldsymbol{p}$, such that for any $v\in\mathcal{C}^0_{\sigma}(\Gamma;\mathcal H^1_0(D))$,
 \begin{align*}
 \|v-\mathcal{I}_{\boldsymbol{p}}v\|_{\mathcal L^2_{\rho}(\Gamma)\otimes \mathcal H^1_0(D)}\lesssim\sum_{n=1}^N\beta_n(p_n)e^{-r_np_n^{\theta_n}},
\end{align*}
where
\begin{itemize}
\item if \ $\Gamma_n$ is bounded \ $
\left\{
\begin{aligned}
&\theta_n = \beta_n  =1 ;\\
&r_n = \log\left[\frac{2\tau_n}{|\Gamma_n|}\left(1+\sqrt{1+\frac{|\Gamma_n|^2}{4(\tau_n)^2}}\right)\right],
\end{aligned}
\right.
$
\vspace{0.3cm}
\item if \ $\Gamma_n$ is unbounded
$
\left\{
\begin{aligned}
&\theta_n = 1/2,\ \ \beta_n = \mathcal{O}(\sqrt{p_n}) ;\\
&r_n = \tau_n\delta_n,
\end{aligned}
\right.
$
\end{itemize}
$\tau_n$ is smaller than the distance between $\Gamma_n$ and the nearest singularity in the complex plane and $\delta_n$ is strictly positive value when $\Gamma_n$ is unbounded such that the joint probability density $\rho$ satisfies
$$
\rho(y)\leq C_{\rho}e^{{-\sum_{n=1}^N}(\delta_ny_n)^2},\quad \forall\  y\in\Gamma,
$$
for some $C_{\rho}>0$. 
\end{lemma}

By Lemma {\ref{interpolation-error}} and (\ref{sobolevinequality}), we have the following result.

\begin{theorem}{\label{random-error}} Let $\mathcal{I}_{\boldsymbol{p}}(u^h)$ be the Lagrange interpolation of $u^h$ with $N_{\boldsymbol{p}}$ collocation points, then, the following result holds
$$
\|u^h-\mathcal{I}_{\boldsymbol{p}}u^h\|_{\mathcal L^2_{\rho}(\Gamma)\otimes \mathcal H_0^1(D)} \lesssim\sum_{n=1}^N\beta_n(p_n)e^{-r_np_n^{\theta_n}},
$$
where parameters $\beta_n, r_n, \theta_n$ are defined in Lemma {\ref{interpolation-error}}.
\end{theorem}

For $\Gamma$ bounded, we introduce 
$$
\tilde{\mathcal{C}}^{\infty}(\Gamma;V)\equiv\left\{\phi\left|\frac{|\Gamma|^{m+1}}{(m+1)!}\left\|\frac{\partial^{m+1}\phi}{\partial y^{m+1}}\right\|_{\mathcal{C}_{\sigma}^0(\Gamma;V)}\right.\leq C\|\phi\|_{\mathcal{C}_{\sigma}^0(\Gamma;V)}, \forall m\in\mathbb{N}\right\},
$$
where $C$ is independent of $m$ and $V$ is a Banach space.

For $\Gamma$ unbounded, we need the following lemma.
\begin{lemma}\label{unboundgamma}
Let $v\in\mathcal{C}_{\sigma}^{0}(\mathbb{R},V)$, we suppose $v$ admits an analytic extension in the strip of the complex plane $\Sigma(\mathbb{R};\tau)=\{z\in\mathbb{C}, {\rm dist}(z,\mathbb{R})\leq\tau\}$ for some $\tau>0$, and 
\begin{align*}
\forall z=(y+\imath w)\in\Sigma(\mathbb{R};\tau),\sigma(y)\|v(z)\|_V\leq C_{v}(\tau),
\end{align*}
 then, for any $\delta>0$, there exist a constant $C$, independent of $p$, and a function $\Theta(p)=\mathcal{O}(\sqrt{p})$ such that
\begin{align}\label{IpuH-3}
\min_{\omega\in\mathcal{P}_p\otimes V}\max_{y\in\mathbb{R}}\left|\|v(y)-\omega(y)\|_Ve^{-\frac{(\delta y)^2}{8}}\right|\leq C\Theta(p)e^{-\tau\delta\sqrt{p}/\sqrt{2}}.
\end{align}
\end{lemma}

\begin{proof}
The proof follows the same procedure as Lemma 4.6 in \cite{babuvska2010stochastic}.
\end{proof}

Next, we given an estimate for the interpolation error in $\mathcal{L}_{\rho}^4(\Gamma)\otimes V$ norm. 
\begin{lemma}\label{interpolation-error-L4}
 There exist positive constants $\tilde{r}_n, n=1,2,\cdots, N$, independent of $\boldsymbol{p}$, such that
 \begin{align*}
 \|v-\mathcal{I}_{\boldsymbol{p}}v\|_{\mathcal L^4_{\rho}(\Gamma)\otimes  V}\lesssim\sum_{n=1}^N\beta_n(p_n)e^{-\tilde{r}_np_n^{\theta_n}},
\end{align*}
for any $v\in\mathcal{C}_{\sigma}^0(\Gamma;V)$ when $\Gamma$ is unbounded and $\tilde{\mathcal{C}}^{\infty}(\Gamma;V)$ when $\Gamma$ is bounded. The parameters 
\begin{equation*}
\tilde{r}_n =
\left\{
\begin{aligned}
&\log\left[\frac{2\tau_n}{|\Gamma_n|}\left(1+\sqrt{1+\frac{|\Gamma_n|^2}{4(\tau_n)^2}}\right)\right],\qquad \;{\rm if\ \Gamma_n\ is\ bounded},\\
&\tau_n\delta_n/\sqrt{2},\qquad\qquad\qquad\qquad\qquad\qquad\ \  {\rm if\ \Gamma_n\ is\ unbounded},
\end{aligned}
\right.
\end{equation*}
and $\theta_n,\beta_n, \tau_n, \delta_n$ are defined the same as Lemma \ref{interpolation-error}.
\end{lemma}

\begin{proof}
By (A1), (A2) and use the same argument 
in \cite{babuvska2010stochastic}, we can prove $\textstyle{\mathcal{C}_{\sigma}^0(\Gamma;V)\subset \mathcal{L}^4_{\rho}(\Gamma;V)}$. 

When $\Gamma$ is bounded, for any $v$ in $\tilde{\mathcal{C}}^{\infty}(\Gamma; V)$, using the Lagrange remainder formula, we have
\begin{align*}
\|\mathcal{I}_{\boldsymbol{p}}v\|_{\mathcal{C}_{\sigma}^0(\Gamma;V)}&\leq\|\mathcal{I}_{\boldsymbol{p}}v-v\|_{\mathcal{C}_{\sigma}^0(\Gamma;V)}+\|v\|_{\mathcal{C}_{\sigma}^0(\Gamma;V)}\\
& \leq(C+1)\|v\|_{\mathcal{C}_{\sigma}^0(\Gamma;V)},
\end{align*}
which implies
\begin{align*}
\|\mathcal{I}_{\boldsymbol{p}}v\|_{\mathcal{L}_{\rho}^4(\Gamma;V)}
& \lesssim \|v\|_{\mathcal{C}_{\sigma}^0(\Gamma;V)}.
\end{align*}
Moreover, since $\textstyle{\mathcal{I}_{\boldsymbol{p}}\omega = \omega, \forall \omega \in \mathcal{P}_{\boldsymbol{p}}(\Gamma)\otimes V}$, the following estimate holds for any $\textstyle{v\in\tilde{\mathcal{C}}^{\infty}(\Gamma;V)}$:
\begin{align}\nonumber
\|v-\mathcal{I}_{\boldsymbol{p}}v\|_{\mathcal{L}^4_{\rho}(\Gamma;V)}&\leq\|v-\omega\|_{\mathcal{L}^4_{\rho}(\Gamma;V)}+\|\mathcal{I}_{\boldsymbol{p}}(\omega-v)\|_{\mathcal{L}^4_{\rho}(\Gamma;V)}\\ \nonumber
&\lesssim \|v-\omega\|_{\mathcal{C}^0_{\sigma}(\Gamma;V)}+\|\omega-v\|_{\mathcal{C}^0_{\sigma}(\Gamma;V)}\\ \label{IpuH-1}
&\lesssim \|v-\omega\|_{\mathcal{C}^0_{\sigma}(\Gamma;V)}.
\end{align} 
By the one-dimensional argument in \cite{babuvska2010stochastic} and Lemma 4.4 therein, for any $v\in \tilde{\mathcal{C}}^\infty(\Gamma;V)$ we get
$$
\|v-\mathcal{I}_{\boldsymbol{p}}v\|_{\mathcal{L}^4_{\rho}(\Gamma;V)}\lesssim 
\min_{w\in \mathcal{P}_{\boldsymbol{p}}\otimes V} \|v-w\|_{C^0_\sigma(\Gamma;V)}
\lesssim \sum_{n=1}^N e^{-p_n\log(\rho_n)},
$$
where $\textstyle{1<\rho_n=\frac{2\tau_n}{|\Gamma_n|}(1+\sqrt{1+\frac{|\Gamma_n|^2}{4\tau_n^2}}})$

For the case when $\Gamma$ is unbounded, following \cite{babuvska2010stochastic}, we can prove that $\mathcal{I}_{\boldsymbol{p}}$ is also a bounded operator from $\textstyle{\mathcal{C}_{\sigma}^0(\Gamma;V)\rightarrow \mathcal{L}^4_{\rho}(\Gamma;V)}$. By Lemma \ref{unboundgamma}, for any $v\in \mathcal{C}^0_\sigma(\Gamma;V)$ we have
\begin{equation}
\|v-\mathcal{I}_{\boldsymbol{p}}v\|_{\mathcal{L}^4_\rho(\Gamma; V)}\lesssim 
\min_{w\in \mathcal{P}_{\boldsymbol{p}}\otimes V} \|v-w\|_{C^0_G(\Gamma;V)}
\lesssim \sum_{n=1}^N \Theta(p_n)e^{-\tau_n\delta_n\sqrt{p_n}/\sqrt{2}},
\label{IpuH-2}
\end{equation}
where $G(y)=\Pi_{n=1}^N G_n(y_n)$ and $G_n(y_n)=e^{-(\delta_ny_n)^2/8}$.

%

Using the isomorphic property between $\mathcal{L}^4_{\rho}(\Gamma,V)$ and $\mathcal{L}^4_{\rho}(\Gamma)\otimes V$, the conclusion follows.
\end{proof}

From the above lemmas, we have the following error estimate for the semi-discrete solutions.

\begin{theorem}\label{two-grid-error} Let $u^h$ be the semi-discrete solution satisfying (\ref{Lweakform}) and $u_h$ be the semi-discrete solution of equation (\ref{weakformula}). Then,
\begin{align*}
\|u_h-u^h\|_{\mathcal L^2_{\rho}(\Gamma)\otimes \mathcal H_0^1(D)}\lesssim H^4 +\sum_{n=1}^N\beta^2_n(P_n)e^{-2\tilde{r}_nP_n^{\theta_n}},
\end{align*}
\end{theorem}
where $\tilde{r}_n$ ($n=1,\dots, N$) are given in Lemma \ref{interpolation-error-L4} .
\begin{proof}

Choose $v=u_{H,\boldsymbol{P}}$ in $(\ref{Lweakform})$ and subtract (\ref{Lweakform}) from (\ref{weakformula}), we have, for a.e. $y\in\Gamma$ and for any $w\in\mathcal{X}_h(D)$, the following equation holds.

\begin{align*}
(a_N(y)\nabla(u_h-u^h)(y),\nabla w) + (f(u_h(y))-f(u_{H,\boldsymbol{P}}(y))-f'(u_{H,\boldsymbol{P}}(y))(u^h-u_{H,\boldsymbol{P}})(y),w) = 0,
\end{align*}
or 
\begin{align*}
A_{u_{H,\boldsymbol{P}}(y)}((u_h-u^h)(y),w) = (\beta(u_h(y)-u_{H,\boldsymbol{P}}(y))^2,w), \end{align*}
where 
$$
\beta = -\int_0^1(1-t)f''(u_{H,\boldsymbol{P}}(y)+t(u_h-u_{H,\boldsymbol{P}})(y))dt.
$$
By assumption, it is easy to see that $\beta$ is a uniformly bounded function on $\bar{D}$. From the H{\"o}lder inequality and (\ref{sobolevinequality}), we get
\begin{align}
(\beta(u_h(y)-u_{H,\boldsymbol{P}}(y))^2,w)\lesssim\|(u_h(y)-u_{H,\boldsymbol{P}}(y))^2\|_{0,\frac{p'}{2}}\|w\|_{0,\frac{p'}{p'-2}}\lesssim\|u_h(y)-u_{H,\boldsymbol{P}}(y)\|_{0,p'}^2\|w\|_1,
\end{align}
where $2\leq p'\leq \infty$ when $d=2$ and $p'=12/5$ when $d=3$.

Applying Lemma \ref{quasi-linear-operator}, for a.e. $y\in\Gamma$, we have

\begin{align*}
\|(u_h-u^h)(y)\|_{1}&\lesssim\sup_{w\in \mathcal X_h(D)}\frac{A_{u_{H,\boldsymbol{P}}(y)}((u_h-u^h)(y),w)}{\|w\|_1} = \sup_{w\in \mathcal X_h(D)}\frac{(\beta((u_h-u_{H,\boldsymbol{P}})(y))^2,w)}{\|w\|_1}\\
&\lesssim\|(u_h-u_{H,\boldsymbol{P}})(y)\|^2_{0,p'}\\
&\lesssim(\|(u_h-u_H)(y)\|_{0,p'}+\|(u_H-u_{H,\boldsymbol{P}})(y)\|_{0,p'})^2.
\end{align*}

By Cauchy-Schwartz inequality, we get
\begin{align*}
\|u_h-u^h\|^2_{\mathcal L^2_{\rho}(\Gamma)\otimes\mathcal H^1_0(D)} &=\int_{\Gamma}\int_D|\nabla(u_h-u^h)|^2dx\rho dy = \int_{\Gamma}\|u_h-u^h\|_{1}^2\rho dy\\
&\lesssim\int_{\Gamma}(\|u_h-u_{H}\|_{0,p'}+\|u_H-u_{H,\boldsymbol{P}}\|_{0,p'})^4\rho dy\\
&\lesssim\int_{\Gamma}\left(\|u_h-u_{H}\|^4_{0,p'}+\|u_H-u_{H,\boldsymbol{P}}\|^4_{0,p'}\right)\rho dy.
\end{align*}
We get
\begin{align}
\|u_h-u^h\|^2_{\mathcal L^2_{\rho}(\Gamma)\otimes\mathcal H^1_0(D)}&\lesssim\|u_h-u_{H}\|^4_{\mathcal{L}_{\rho}^4(\Gamma)\otimes\mathcal{L}^{2p'}(D)}+\|u_H-u_{H,\boldsymbol{P}}\|^4_{\mathcal{L}_{\rho}^4(\Gamma)\otimes\mathcal{L}^{2p'}(D)}.
\label{proof-final-formula}
\end{align}

By Lemma \ref{error4fem} and denote $p=2p'$ we have
\begin{align}\label{uhuHerror}
\|u_h-u_H\|_{\mathcal L^4_{\rho}(\Gamma)\otimes \mathcal L^{p}(D)}\lesssim H^2\|u\|_{\mathcal L^4_{\rho}(\Gamma)\otimes \mathcal W^{2,p}(D)}.
\end{align}

Notice $\|u_H-u_{H,\boldsymbol{P}}\|_{\mathcal L^4_{\rho}(\Gamma)\otimes \mathcal L^p(D)} = \|u_H-\mathcal{I}_{\boldsymbol{P}}(u_{H})\|_{\mathcal L^4_{\rho}(\Gamma)\otimes \mathcal L^p(D)}$, from (\ref{proof-final-formula}), (\ref{uhuHerror}) and Lemma \ref{interpolation-error-L4}, we have
$$
\|u_h-u^h\|_{\mathcal L^2_{\rho}(\Gamma)\otimes \mathcal H^1_0(D)}\lesssim H^4+\sum_{n=1}^N\beta^2_n(P_n)e^{-2\tilde{r}_nP_n^{\theta_n}}.
$$
\end{proof}
Finally, we get the following error estimates for the two-level solution.

\begin{theorem}
Let $u^{h,\boldsymbol{p}}$ be the two-level solution and $u$ be the exact solution of (\ref{KLmodel-problem}). Then, we have
\begin{align}\label{H1error}
\|u-u^{h,\boldsymbol{p}}\|_{\mathcal L^2_{\rho}(\Gamma)\otimes \mathcal H^1_0(D)}&\lesssim h+H^4+\sum_{n=1}^N\beta^2_n(P_n)e^{-2\bar{r}_nP_n^{\theta_n}}+\sum_{n=1}^N\beta_n(p_n)e^{-\bar{r}_np_n^{\theta_n}},\\
\label{L2error}
\|u-u^{h,\boldsymbol{p}}\|_{\mathcal L^2_{\rho}(\Gamma)\otimes \mathcal L^2(D)}&\lesssim h^2+H^4+\sum_{n=1}^N\beta^2_n(P_n)e^{-2\bar{r}_nP_n^{\theta_n}}+\sum_{n=1}^N\beta_n(p_n)e^{-\bar{r}_np_n^{\theta_n}}.
\end{align}
where $\bar{r}_n=\min\{\tilde{r}_n, r_n\}$ and
$\beta_n,\ \theta_n$ are constants from Lemma \ref{interpolation-error}.
\end{theorem}
\begin{proof}
Estimation (\ref{H1error}) follows from Lemma \ref{error4fem} and \ref{interpolation-error}, Theorem \ref{random-error} and \ref{two-grid-error}. For ($\ref{L2error}$), we have
\begin{align*}
\|u-u^{h,\boldsymbol{p}}\|_{\mathcal L_{\rho}^2(\Gamma)\otimes \mathcal L^2(D)}&\lesssim \|u-u_h\|_{\mathcal L_{\rho}^2(\Gamma)\otimes \mathcal L^2(D)}+\|u_h-u^h\|_{\mathcal L_{\rho}^2(\Gamma)\otimes \mathcal L^2(D)}+\|u^h-\mathcal{I}_{\boldsymbol{p}}u^h\|_{\mathcal L_{\rho}^2(\Gamma)\otimes \mathcal L^2(D)}\\
&\lesssim  \|u-u_h\|_{\mathcal L_{\rho}^2(\Gamma)\otimes \mathcal L^2(D)}+\|u_h-u^h\|_{\mathcal L_{\rho}^2(\Gamma)\otimes \mathcal H_0^1(D)}+\|u^h-\mathcal{I}_{\boldsymbol{p}} u^h\|_{\mathcal L_{\rho}^2(\Gamma)\otimes \mathcal L^2(D)}\\
&\lesssim h^2+H^4+\sum_{n=1}^N\beta^2_n(P_n)e^{-2\bar{r}_nP_n^{\theta_n}}+\sum_{n=1}^N\beta_n(p_n)e^{-\bar{r}_np_n^{\theta_n}}.
\end{align*}
\end{proof}

\section{Numerical experiments}
\label{sec:numericalresult} 
In this section, we present some numerical experiments to verify the theoretical results given in Section \ref{sec:convergence}. Our model problem is
\begin{align*}
-\nabla\cdot(a\nabla u) + u^3 &=g,\quad {\rm in}\ \Gamma\times D,\\
u &= 0,\quad {\rm on} \ \Gamma\times\partial D,
\end{align*}
where $D = (-1, 1)^2$, $\Gamma=(-1, 1)^2$ in Example 1, and $\Gamma=(-1, 1)^4$ in Example 2. We first choose a particular diffusion coefficient such that the exact solution is available. Then, we consider a case with random coefficient given by truncated KL expansion. 

We use piecewise linear finite element method for the spatial discretization. The stopping criterion for Newton iteration is chosen to be the relative error between two adjacent iterates less than a prescribed tolerance, i.e.,
$$
\frac{\|U_H^{l+1}-U_H^{l}\|}{\|U_H^{l+1}\|}\leq\epsilon.
$$ 
The tolerance $\epsilon = 10^{-2}$ is used in our numerical tests reported below, and a tighter tolerance does not yield better overall solution accuracy for these examples. In general, for problems which need tighter Newton tolerance in order to obtain better accuracy, the computational savings of the two-level method will be greater. For the linear system of equations, we use the algebraic multigrid method with tolerance $10^{-9}$. 

The numerical experiments are conducted on a desktop computer with 3.5 GHz 6-core Intel Xeon E5 CPU and 16 GB 1867 MHz DDR3 memory. The MATLAB finite element package iFEM is used for the implementation \cite{long2009ifem}.

{\bf Example $1$}: We choose the following random coefficient
$$
a(Y_1(\omega),Y_2(\omega), x_1, x_2) = 3 + Y_1(\omega) + Y_2(\omega),
$$
where the random variables $Y_n(\omega),\ (n=  1,2)$ are independent and identically distributed, satisfying uniform distribution. The collocation points are zeros of the Legendre polynomials. The right-hand side function $g$ is defined by
$$
g(\omega,x)=2\pi^2 \sin(\pi x_1)\sin(\pi x_2)+\left(\frac{\sin(\pi x_1)\sin(\pi x_2)}{a(Y_1(\omega),Y_2(\omega),x_1,x_2)}\right)^3.
$$
Hence, the exact solution is
$$
u(Y_1(\omega),Y_2(\omega), x_1, x_2) = \frac{1}{a(Y_1,Y_2,x_1,x_2)}\sin(\pi x_1)\sin(\pi x_2).
$$

To see the order of accuracy in physical space, we choose mesh size pairs $(H,h)$=$(\textstyle{\frac{1}{2}}, \textstyle{\frac{1}{4}})$, $(\textstyle{\frac{1}{4}}, \textstyle{\frac{1}{16}})$, $(\textstyle{\frac{1}{8}}, \textstyle{\frac{1}{64}})$, and $(\textstyle{\frac{1}{16}}, \textstyle{\frac{1}{256}})$, and fix the polynomial space pair with $(\boldsymbol{P},\boldsymbol{p}) = (4, 8)$ (note that $(\boldsymbol{P},\boldsymbol{p})=(4,8)$ means $P_n=4, p_n=8$ for $n=1,\cdots,N$) such that the overall approximation error is dominated by the spatial discretization error. Similarly, for the order of accuracy in the stochastic domain, we choose a fixed mesh size pair $(H, h) = (\textstyle{\frac{1}{32}}, \textstyle{\frac{1}{1024}})$ and consider the polynomial space pairs with $(\boldsymbol{P},\boldsymbol{p})=(1,2), (2,4), (3,6), (4,8).$ 

From the left figure of Fig. \ref{conorderex1}, we observe that the convergence order is $\mathcal{O}(h^2)$ (equivalently, $\mathcal{O}(\bold{N}^{-1})$ where $\bold{N}$ is the total number of degrees of freedom in space) in the $\mathcal L^2_{\rho}(\Gamma)\otimes \mathcal L^2(D)$ norm and is $\mathcal{O}(h)$ (equivalently, $\mathcal{O}(\bold{N}^{-0.5})$) in the $\mathcal L^2_{\rho}(\Gamma)\otimes \mathcal H_0^1(D)$ norm which are consistent with the theory.  The right figure of Fig. \ref{conorderex1} shows that the error decays exponentially with respect to the polynomial degree $p$ which is also consistent with the theoretical results. For this example, since the distance between $\Gamma_n$ and the nearest singularity is $\sqrt{2}$, we choose $\tau_n=\sqrt{2}$. From Lemma \ref{interpolation-error}, we get the upper bound $1.6536$ for $r_n$. By using the linear least squares fitting, we get numerical value of $r_n$ as $1.3394$ which is less than the theoretical upper bound.
\begin{figure}
\centering
\includegraphics[width=7.5cm]{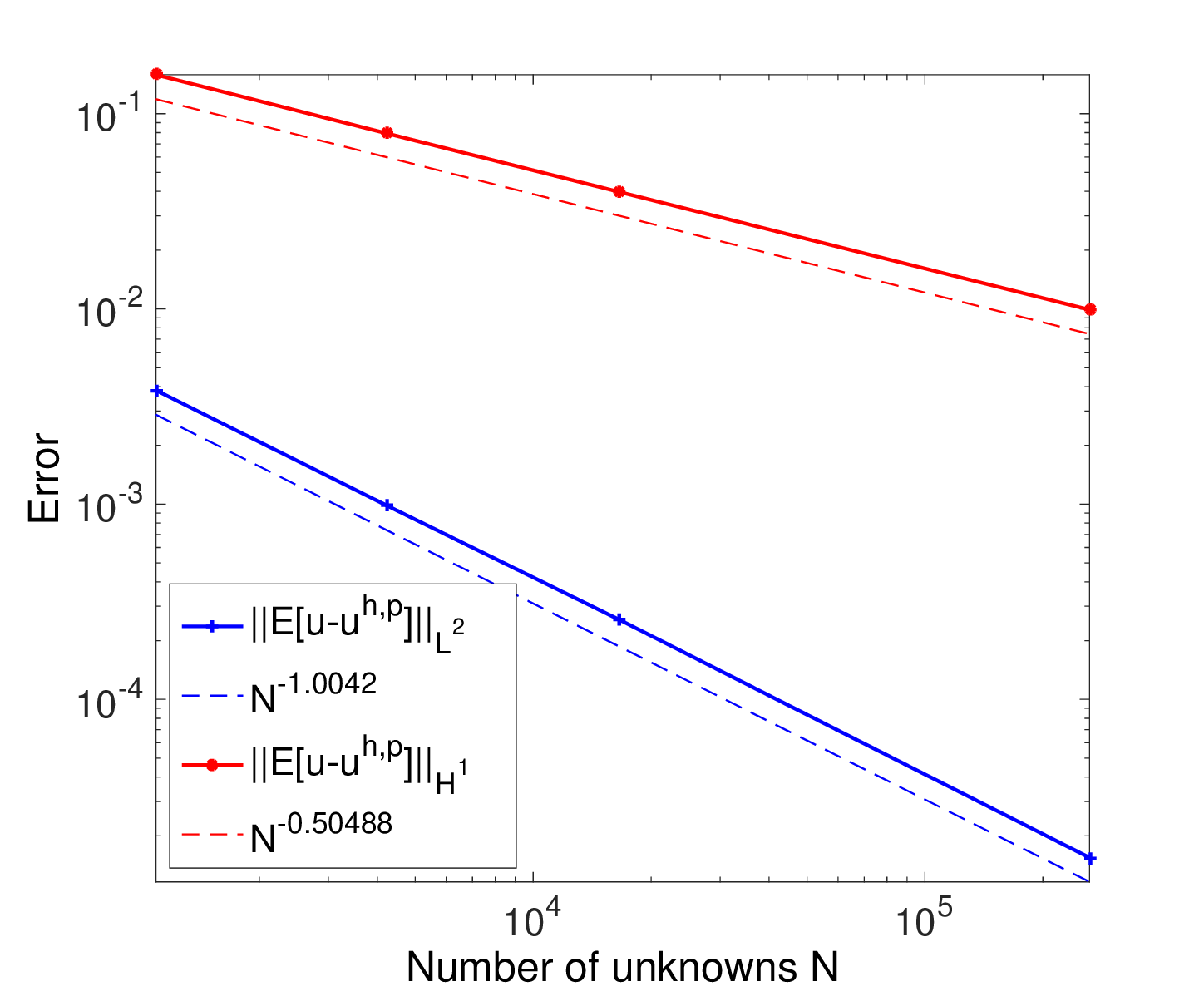}
\includegraphics[width=7.5cm]{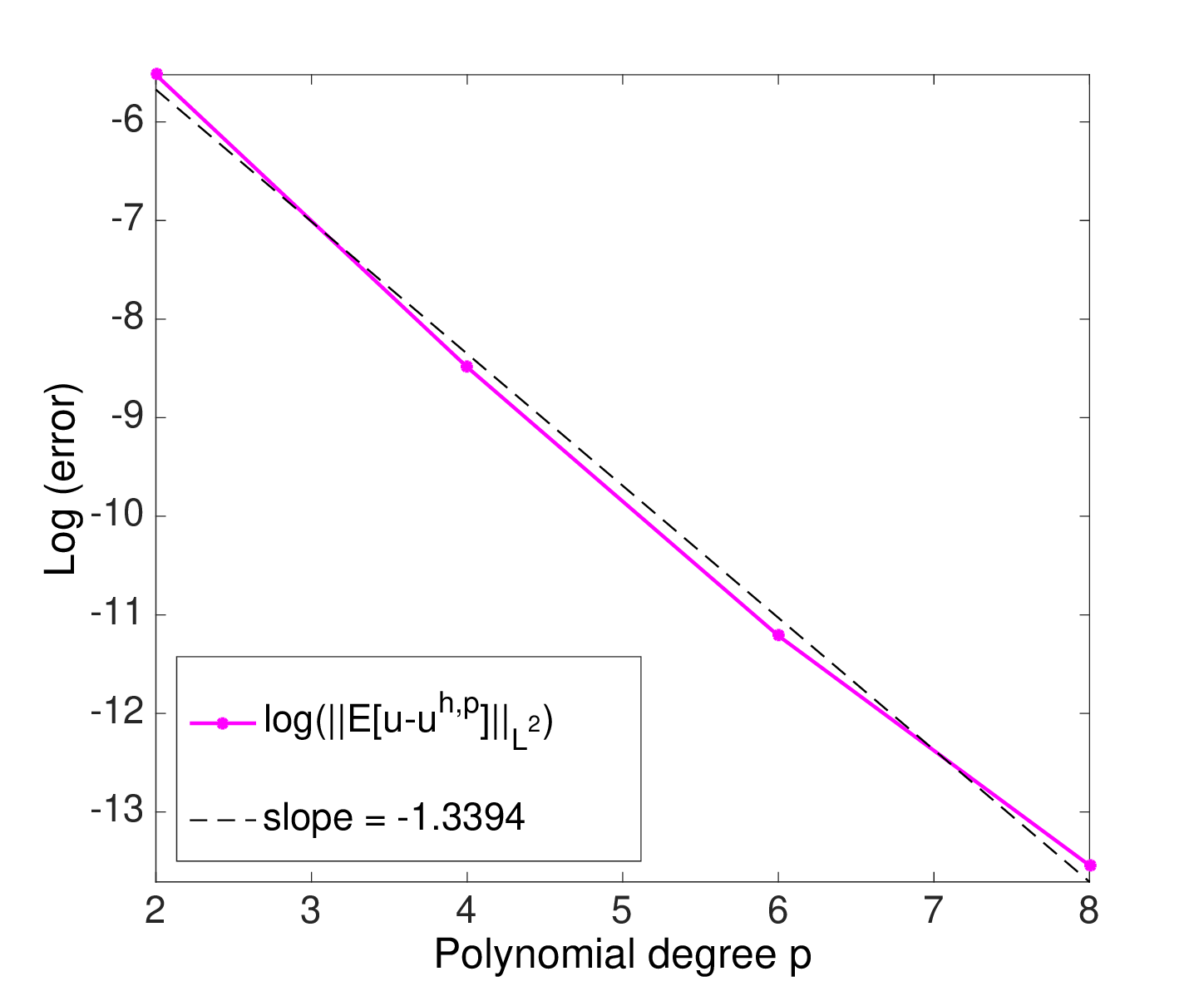}
\caption{Example 1: the convergence rate in physical space (left) 
and random space (right).}
\label{conorderex1}
\vspace{0cm}
\end{figure}
Numerical results presented in Table \ref{H2hex1} show that two-level solution has the same accuracy as the standard stochastic collocation solution using $\mathcal{T}_h(D)$ and $\mathcal{P}_{\boldsymbol{p}}(\Gamma)$ in $\mathcal L_{\rho}^2(\Gamma)\otimes \mathcal L^2(D)$ and $\mathcal L_{\rho}^2(\Gamma)\otimes \mathcal H_0^1(D)$ norms when $h=H^2$. Table \ref{H4hex1} demonstrates that accuracy of the two-level solution is the same as the fine level stochastic collocation solution in $\mathcal L_{\rho}^2(\Gamma)\otimes \mathcal H_0^1(D)$ norm when $h=H^4$. To better understand the accuracy of the solution ($u^{h,\boldsymbol{p}}$) obtained by combining the two-grid finite element discretization in the physical space with the two-level collocation method in the random domain, we also report in Tables \ref{H2hex1} and \ref{H4hex1} results from the solution ($u_{\boldsymbol{P}}^h$) obtained by using low-level collocation for both semilinear equations on coarse mesh and linearized equations on fine mesh, and results from the solution ($u_{\boldsymbol{p}}^h$) obtained by using high-level collocation for both semilinear equations on coarse mesh and linearized equations on fine mesh. It is worth pointing out that the accuracy of the two-level solution is one order of magnitude worse than the fine level stochastic collocation solution in $\mathcal L_{\rho}^2(\Gamma)\otimes \mathcal L^2(D)$ norm in the case $h=H^4$, but not in the case $h=H^2$.


In terms of computational complexity for this example, when using standard stochastic collocation method with $h=1/256$ and $p=8$, we need to solve $170$ ($=73 \times 2 + 8 \times 3$) linear systems of equations with $263169$ unknowns; and for the two-level stochastic collocation method with $(H, h)=(1/16, 1/256)$ and $(P, p)=(4, 8)$, we only solve $81$ linear systems of equations with $263169$ unknowns and $51$ ($=24\times 2 + 1\times 3$) linear systems of equations with $1089$ unknowns. Since the computations are dominated by solving the large linear systems of equations, the actual runtime of the two-level collocation method is less than a half of the runtime for standard collocation method.


\begin{table}[H]
\centering
{\begin{tabular}{||c|c|c||}
\hline
& $\mathcal L^2_{\rho}(\Gamma)\otimes \mathcal L^2(D)-{\rm norm}$ & $\mathcal L^2_{\rho}(\Gamma)\otimes \mathcal H^{1}_0(D)-{\rm norm}$\\
\hline
$ u-u_{H,\boldsymbol{P}} $     & $0.0042$       & $0.1668$ \\[4pt]
\hline
$ u-u_{\boldsymbol{P}}^{h} $   & $3.6165E-4$ & $0.0106$ \\[4pt]
\hline
$ u-u^{h,\boldsymbol{p}} $      & $1.6187E-5$ & $0.0105$ \\[4pt]
\hline
$ u-u_{\boldsymbol{p}}^{h} $   & $1.6182E-5$ & $0.0105$ \\[4pt]
\hline
$ u-u_{h,\boldsymbol{p}} $     & $1.5940E-5$ & $0.0105$ \\[4pt]
\hline
\end{tabular}}
\vspace{0.3cm}
\caption{Example 1: approximation errors with $(H, h) = (1/16,1/256),\ (\boldsymbol{P},\boldsymbol{p}) =(4,8)$}
\label{H2hex1}
\end{table}

\begin{table}[H]
\centering
{\begin{tabular}{||c|c|c||}
\hline
& $\mathcal L^2_{\rho}(\Gamma)\otimes \mathcal L^2(D)-{\rm norm}$ & $\mathcal L^2_{\rho}(\Gamma)\otimes \mathcal H^{1}_0(D)-{\rm norm}$\\
\hline
$ u-u_{H,\boldsymbol{P}} $  & $0.0167$ & $0.6417$ \\[4pt]
\hline
$ u-u_{\boldsymbol{P}}^{h} $ & $3.7737E-4$ & $0.0106$ \\[4pt]
\hline
$ u-u^{h,\boldsymbol{p}} $   & $1.2760E-4$ & $0.0105$ \\[4pt]
\hline
$ u-u_{\boldsymbol{p}}^{h} $   & $1.2672E-4$ & $0.0105$ \\[4pt]
\hline
$ u-u_{h,\boldsymbol{p}} $  & $1.5960E-5$ & $0.0105$ \\[4pt]
\hline
\end{tabular}}
\vspace{0.3cm}
\caption{Example 1: approximation errors with $(H, h)=(1/4, 1/256),\ (\boldsymbol{P},\boldsymbol{p}) =(4,8)$}
\label{H4hex1}
\end{table}

{\textbf{Example $2$}}: We choose
\begin{align*}
a(\omega, x) =  \bar{a}(x)+\sum_{n=1}^4\sqrt{\lambda_n}b_n(x)Y_n(\omega),
\end{align*}
where $\rho = 0.25$, $\bar{a}(x) = 1$, $\lambda_n, b_n(x)$ are the eigenvalues and eigenfunctions corresponding to the covariance function $\text{Cov}_a(x, x')=\sigma^2 \exp(-|x-x'|)$ with $\sigma = 0.6$. We choose $g=2|x|^2-1$.

Since the exact solution of this example is not available, we construct a ``reference" solution numerically by using a very fine mesh ($h = 1/512$) and a polynomial space of high degree ($\boldsymbol{p}=9$) and denote it by $u^*= u_{h,\boldsymbol{p}}$.  For the order of accuracy in the random space, we choose a fixed mesh size pair $(H, h) = (\textstyle{\frac{1}{16}}, \textstyle{\frac{1}{256}})$ and consider the polynomial space pairs with $(\boldsymbol{P},\boldsymbol{p})=(1,2), (2,4), (3,6).$ To see the convergence order in physical space, we fixed polynomial space pair $\textstyle{(\boldsymbol{P},\boldsymbol{p})=(4,8)}$ and choose mesh size pairs $(H,h)$=$(\textstyle{\frac{1}{2}}, \textstyle{\frac{1}{4}})$, $(\textstyle{\frac{1}{4}}, \textstyle{\frac{1}{16}})$, $(\textstyle{\frac{1}{8}}, \textstyle{\frac{1}{64}})$, and $(\textstyle{\frac{1}{16}}, \textstyle{\frac{1}{256}})$. To compute the convergence order in physical space, we use the error between the two adjacent mesh size pairs. It can be seen from the left figure of Fig. {\ref{conorderex2}} that the accuracy is of the optimal order (i.e., $\mathcal{O}(\bold{N}^{-1})$) in $\mathcal L^2_{\rho}(\Gamma)\otimes \mathcal L^2(D)$ norm, as well as in $\mathcal L^2_{\rho}(\Gamma)\otimes \mathcal H_0^1(D)$ norm, (i.e. $\mathcal{O}(\bold{N}^{-0.5})$). 
The right  figure of Fig. {\ref{conorderex2}} shows the exponential decay with respect to the  polynomial degree $p$. 
\begin{figure}
\centering
\includegraphics[width=7.5cm]{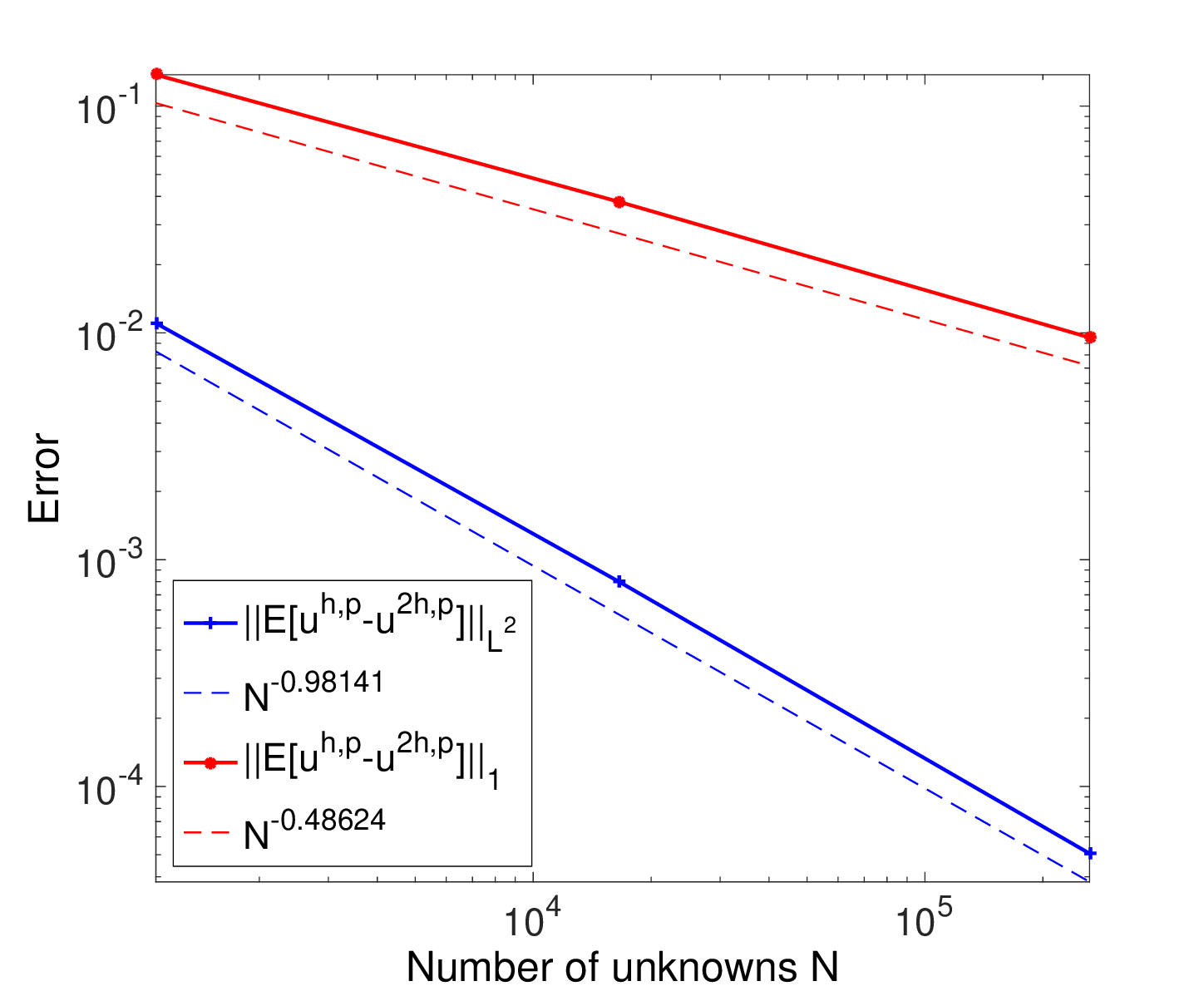}
\includegraphics[width=7.5cm]{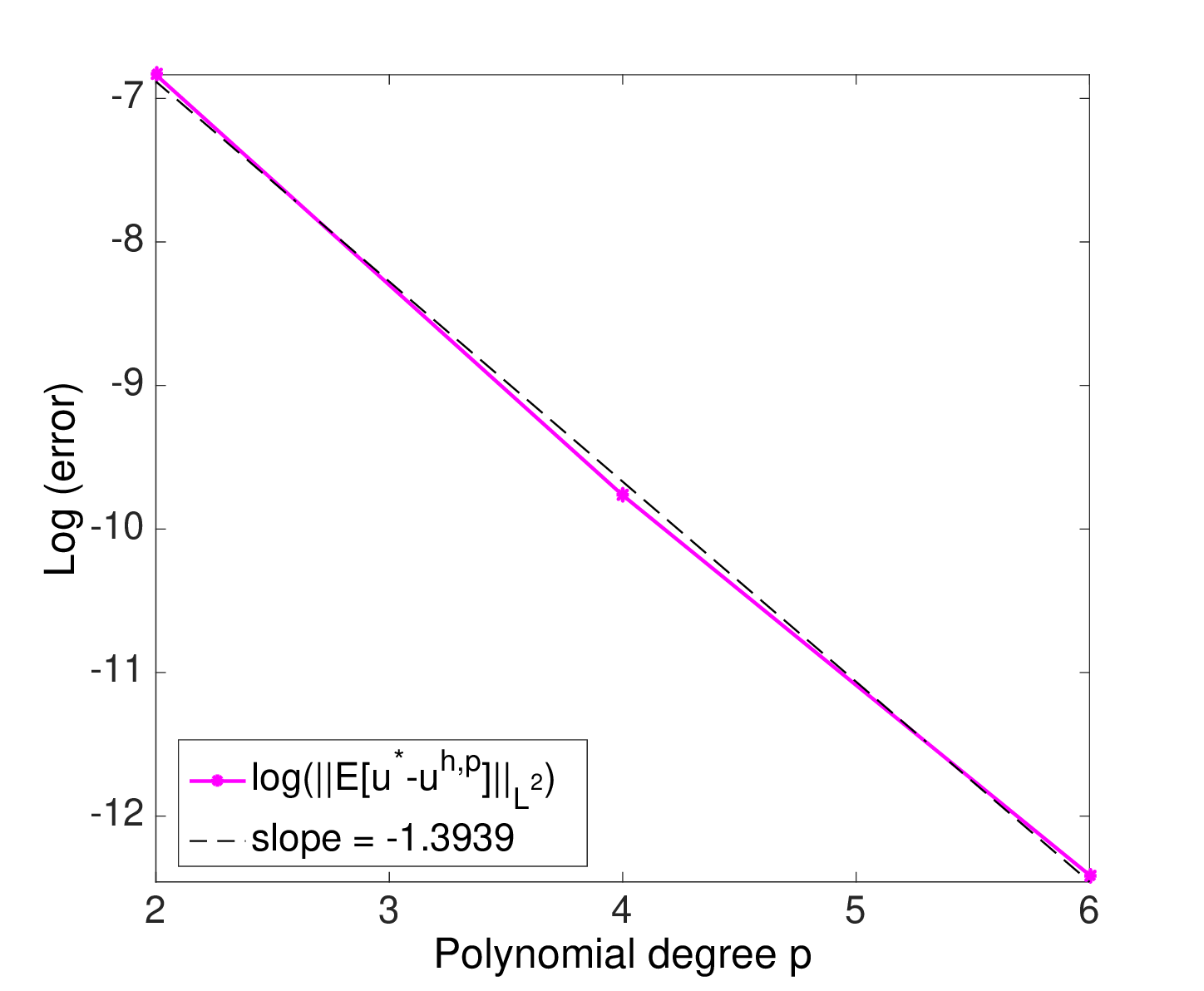}
\caption{Example 2: the convergence rate in physical space (left) and random space (right)}
\label{conorderex2}
\end{figure}
\begin{table}[H]
\centering
{\begin{tabular}{||c|c|c||}
\hline
& $\mathcal L^2_{\rho}(\Gamma)\otimes \mathcal L^2(D)-{\rm norm}$ & $\mathcal L^2_{\rho}(\Gamma)\otimes \mathcal H^{1}_0(D)-{\rm norm}$\\
\hline
$ u^*-u_{H,\boldsymbol{P}} $  & $1.6737E-4$ & $0.0014$ \\[4pt]
\hline
$ u-u_{\boldsymbol{P}}^{h} $   & $5.7462E-5$ & $2.6445E-4$ \\[4pt]
\hline
$ u^*-u^{h,\boldsymbol{p}} $   & $4.8794E-7$ & $7.1205E-6$ \\[4pt]
\hline
$ u-u_{\boldsymbol{p}}^{h} $   & $4.8782E-7$ & $7.1209E-6$ \\[4pt]
\hline
$ u^*-u_{h,\boldsymbol{p}} $  & $4.9333E-7$ & $7.1234E-6$ \\[4pt]
\hline
\end{tabular}}
\vspace{0.3cm}
\caption{Example 2: approximation errors with $(H,h)=(1/16,{1}/{256}),\ (\boldsymbol{P},\boldsymbol{p}) =(4,8)$}
\label{H2hex2}
\end{table}

Same conclusions can be drawn from Table \ref{H2hex2} and Table \ref{H4hex2} as those for Example 1. 
\begin{table}[H]
\centering
{\begin{tabular}{||c|c|c||}
\hline
& $\mathcal L^2_{\rho}(\Gamma)\otimes \mathcal L^2(D)-{\rm norm}$ & $\mathcal L^2_{\rho}(\Gamma)\otimes \mathcal H^{1}_0(D)-{\rm norm}$\\
\hline
$ u^*-u_{H,\boldsymbol{P}} $  & $0.0019$ & $0.0103$ \\[4pt]
\hline
$ u-u_{\boldsymbol{P}}^{h} $   & $5.8569E-5$ & $2.6728E-4$ \\[4pt]
\hline
$ u^*-u^{h,\boldsymbol{p}} $   & $3.5925E-6$ & $1.1866E-5$ \\[4pt]
\hline
$ u-u_{\boldsymbol{p}}^{h} $   & $3.6285E-6$ & $1.1950E-5$ \\[4pt]
\hline
$ u^*-u_{h,\boldsymbol{p}} $  & $4.9333E-7$ & $7.1234E-6$ \\[4pt]
\hline
\end{tabular}}
\vspace{0.3cm}
\caption{Example 2: approximation errors with $(H,h)=(1/4,{1}/{256}),\ (\boldsymbol{P},\boldsymbol{p})=(4,8)$}
\label{H4hex2}
\end{table}


\section{Conclusion and future work}
We study the stochastic collocation method for solving semilinear elliptic equation with random coefficients. A novel two-level discretization technique is proposed to improve the efficiency of the standard stochastic collocation method. We analyze the convergence of this two-level discretization scheme and prove that when choosing the discretization parameters $h, H, \boldsymbol{p}, \boldsymbol{P}$ appropriately, the two-level solution has the same order of accuracy as the fine level stochastic collocation solution. We also verify the theoretical results by several numerical examples. The main advantage of the two-level approach is that it reduces the computational complexity significantly. 

In this paper we focus on solving stochastic semilinear elliptic problems, however, the two-level method is applicable to general quasilinear partial differential equations with random coefficients \cite{Barajas-Solano:2016aa}. For problems with high stochastic dimensions, we may use sparse grid stochastic collocation method as opposed to the full tensor-product collocation used in this work. It may also be possible to generalize the idea of the two-level collocation method to utilize multiple levels of stochastic spaces when solving stochastic partial differential equations. These will be left as future work.

\subsection*{Acknowledgements} 
L. Chen is supported by the Fundamental Research Funds for the Central Universities of China (2682016CX108) and National Natural Science Foundation of China under Grant No. 11501473. G. Lin and B. Zheng would like to acknowledge the support by the Applied Mathematics Program within the Department of Energy Office of Advanced Scientific Computing Research as part of the Modeling and Simulation of High Dimensional Stochastic Multiscale PDE Systems project and the Collaboratory on Mathematics for Mesoscopic Modeling of Materials project. N. Voulgarakis was supported by the National Science Foundation under Grant No. DMS 1418962. Computations were performed using the computational resources of Pacific Northwest National Laboratory (PNNL) Institutional Computing cluster systems and the National Energy Research Scientific Computing Center at Lawrence Berkeley National Laboratory. The PNNL is operated by Battelle for the US Department of Energy under Contract DE-AC05-76RL01830.


\section*{References}
\bibliographystyle{spmpsci}      
\bibliography{reflibrary4SPDE}

\end{document}